\documentclass{amsart}
\pdfoutput=1
\usepackage{amssymb}
\usepackage[utf8]{inputenc}
\usepackage{xcolor}
\usepackage{tikz}
\usepackage{tikz-cd}
\usetikzlibrary{decorations.pathmorphing,shapes,calc}
\usepackage{fullpage}
\usepackage{enumitem}

\newtheorem{thm}{Theorem}[section]

\newtheorem{lem}[thm]{Lemma}

\newtheorem{cor}[thm]{Corollary}

\theoremstyle{definition}

\theoremstyle{remark}

\newtheorem{rem}[thm]{Remark}

\newcommand{\floor}[1]{\left\lfloor#1\right\rfloor}

\newcommand{\trop}{\mathrm{trop}}

\newcommand{\Z}{\mathbb{Z}}

\newcommand{\R}{\mathbb{R}}
\newcommand{\U}{\mathcal{U}}
\renewcommand{\P}{\mathbb{P}}
\newcommand{\M}{M}

\newcommand{\abs}[1]{\left\lvert#1\right\rvert}

\newcommand{\Mbar}{\overline{M}}
\newcommand{\into}{\hookrightarrow}

\title{Cross-ratio degrees and triangulations}
\author[Silversmith]{Rob Silversmith}
\address{
  \begin{tabular}{l}
     Rob Silversmith\\
     Mathematics Institute, University of Warwick, Coventry CV4 7AL, UK
\end{tabular}}
\email{Rob.Silversmith@warwick.ac.uk}

\begin{document}

\begin{abstract}
    The cross-ratio degree problem counts configurations of $n$ points on $\P^1$ with $n-3$ prescribed cross-ratios. Cross-ratio degrees arise in many corners of combinatorics and geometry, but their structure is not well-understood in general. Interestingly, examining various special cases of the problem can yield combinatorial structures that are both diverse and rich. In this paper we prove a simple closed formula for a class of cross-ratio degrees indexed by triangulations of an $n$-gon; these degrees are connected to the geometry of the real locus of $\M_{0,n}$, and to positive geometry.
\end{abstract}


\maketitle

%%%%%%%%%%%INTRO
\section{Introduction}
Consider a regular $n$-gon $X$ with edges labeled by $[n]=\{1,\ldots,n\}$, in order. To each diagonal $D$ of $X$ is naturally assigned a 4-element subset of $[n]$, consisting of the four edges that $D$ touches. A triangulation $T$ of $X$ is a choice of $n-3$ diagonals in $X$ that do not cross. Thus we may associate to $T$ a collection $\U=\{S_1,\ldots,S_{n-3}\}$ of 4-elements subsets of $[n]$, see Figure \ref{fig:ExampleTriangulation}.

Let $\M_{0,n}=\M_{0,[n]}$ denote the moduli space of configurations of $n$ distinct points on the Riemann sphere, labeled by $[n]$, up to M\"obius transformation. Recall that $\M_{0,[n]}$ is a smooth $(n-3)$-dimensional affine variety, and that for $S\subseteq[n]$ with $\abs{S}\ge3$, there is a forgetful map $\M_{0,[n]}\to\M_{0,S}.$ By the previous paragraph, a triangulation $T$ of a regular $n$-gon defines a product of forgetful maps 
\begin{align}\label{eq:PiT}
    \pi_T:\M_{0,[n]}\to\prod_{S\in\U}\M_{0,S}.
\end{align} Here $\M_{0,S}\cong\P^1\setminus\{\infty,0,1\}$ by taking the cross-ratio. Note that $\pi_T$ is a map of $(n-3)$-dimensional varieties, hence a general fiber of $\pi_T$ is zero-dimensional with some constant cardinality $d_T$ --- this cardinality is an example of a \emph{cross-ratio degree} in the sense of \cite{Silversmith2021}. The purpose of this brief article is to prove:
\begin{thm}\label{thm:main}
    $d_T=2^{I(T)},$ where $I(T)$ is the number of triangles of $T$ with no exterior edges.
\end{thm}

\begin{figure}
\centering
    \begin{tikzpicture}[scale=.75]
        \foreach \x in {1,...,13} {
        \draw (360*\x/13+3*360/52:3)--(360*\x/13+360/13+3*360/52:3);
        \draw (360*\x/13+3*360/52:3) node {$\bullet$};
        \draw (-360*\x/13+360*4/13+360/52:3.3) node {\small \x};
        }
        \filldraw[green,opacity=.05] (-360*0/13+15*360/52:3)--(-360*3/13+15*360/52:3)--(-360*9/13+15*360/52:3)--cycle;
        \filldraw[blue,opacity=.05] (-360*0/13+15*360/52:3)--(-360*11/13+15*360/52:3)--(-360*9/13+15*360/52:3)--cycle;
        \filldraw[orange,opacity=.05] (-360*3/13+15*360/52:3)--(-360*5/13+15*360/52:3)--(-360*8/13+15*360/52:3)--cycle;
        \draw (-360*0/13+15*360/52:3)--(-360*3/13+15*360/52:3);
        \draw[red,thick] (-360*0/13+15*360/52:3)--(-360*2/13+15*360/52:3);
        \draw (-360*0/13+15*360/52:3)--(-360*11/13+15*360/52:3);
        \draw (-360*3/13+15*360/52:3)--(-360*9/13+15*360/52:3);
        \draw (-360*8/13+15*360/52:3)--(-360*3/13+15*360/52:3);
        \draw (-360*5/13+15*360/52:3)--(-360*8/13+15*360/52:3);
        \draw (-360*3/13+15*360/52:3)--(-360*5/13+15*360/52:3);
        \draw (-360*5/13+15*360/52:3)--(-360*7/13+15*360/52:3);
        \draw (-360*0/13+15*360/52:3)--(-360*9/13+15*360/52:3);
        \draw (-360*9/13+15*360/52:3)--(-360*11/13+15*360/52:3);
        \draw[->,
line join=round,
decorate, decoration={
    snake,
    segment length=8,
    amplitude=1.9,post=lineto,
    post length=2pt
}] (4,0)--(5,0);
        \draw (7,0) node {$
        \begin{array}{c}
        \textcolor{red}{\{1,2,3,13\}}\\
        \{1,3,4,13\}\\
        \{1,9,10,13\}\\
        \{1,11,12,13\}\\
        \{3,4,5,6\}\\
        \{3,4,8,9\}\\
        \{3,4,9,10\}\\
        \{5,6,7,8\}\\
        \{5,6,8,9\}\\
        \{9,10,11,12\}
        \end{array}
        $};
        \draw[->,
line join=round,
decorate, decoration={
    snake,
    segment length=8,
    amplitude=1.9,post=lineto,
    post length=2pt
}] (9,0)--(10,0);
        \draw (14,1) node {$\M_{0,[13]}$};
        \draw[->] (14,.5)--(14,-.5);
        \draw (13.5,0) node {$\pi_T$};
        \draw (14,-1) node {$\textcolor{red}{\M_{0,\{1,2,3,13\}}}\times\cdots\times\M_{0,\{9,10,11,12\}}$};
    \end{tikzpicture}
    \caption{The subsets $S_1,\ldots,S_{10}\subseteq[13]$ associated to a  triangulation of a 13-gon, and the corresponding map of moduli spaces. For illustration, an edge and its corresponding subset/factor are colored red. The three ``internal'' triangles are shaded --- Theorem \ref{thm:main} implies $\pi_T$ has degree $d_T=2^3=8.$}
    \label{fig:ExampleTriangulation}
\end{figure}
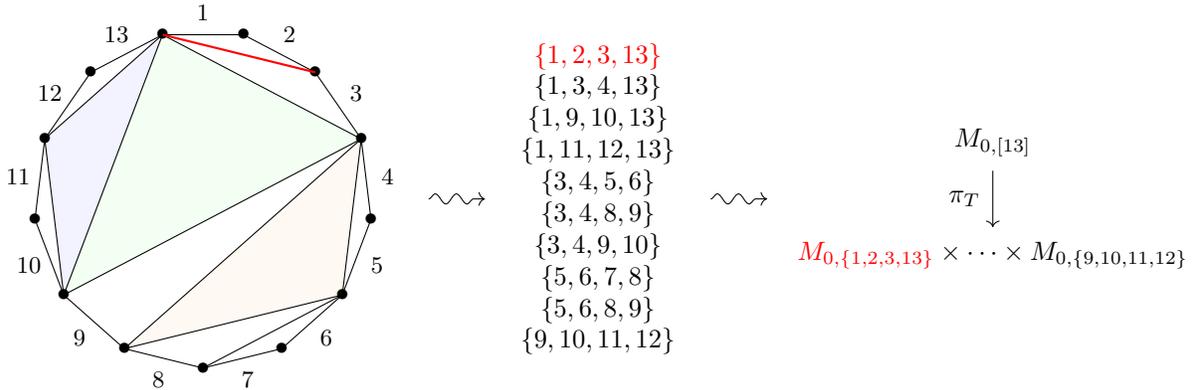

\subsection*{Motivation}
The more general ``cross-ratio degree problem'' is as follows. Let $S_1,\ldots,S_{n-3}\subseteq[n]$ with $\abs{S_j}=4$ for $1\le j\le n-3$, and let $\U=\{S_1,\ldots,S_{n-3}\}.$ The map $\pi_{\U}:\M_{0,[n]}\to\prod_{j=1}^k\M_{0,S_j}$ is a map of $(n-3)$-dimensional varieties, hence is generically finite. The dependence of the degree of $\pi_{\U}$ on the combinatorial data $\U$ is not well-understood.

The particular type of cross-ratio degree appearing in Theorem \ref{thm:main} arises in connection with the \emph{positive geometry} of $\M_{0,n}$ \cite{ArkaniHamedBaiLam2017}, which has recently been applied to a range of important applications. Brown \cite{Brown2009} studied coordinate systems arising from cross-ratios of this form under the name \emph{dihedral coordinates}, and used them to give affine charts on $\Mbar_{0,n}$ that are particularly well-behaved with respect to the cell decomposition of $\Mbar_{0,n}(\R)$, ultimately proving that all period integrals of $\Mbar_{0,n}$ are multiple zeta values. 

More recently, dihedral coordinates have been used extensively as a computational and conceptual tool in string theory, particularly in the study of \emph{tree-level string scattering amplitudes}. Dihedral coordinates satisfy a remarkable collection of relations known as the \emph{$u$-equations} that appear in the study of scattering amplitudes, and which exhibit $\M_{0,n}$ as a \emph{binary geometry} \cite{ArkaniHamedHeLamThomas2023}. One can give natural expressions in terms of dihedral coordinates for the two key objects in the Cachazo-He-Yuan formalism for ``gluon tree amplitudes in pure Yang-Mills theory'' \cite{CachazoHeYuan2014,BaadsgaardBjerrumBohrBourjailyDamgaardFeng2015,DolanGoddard2014}; the \emph{scattering equations} and the \emph{Parke-Taylor form}, a canonical top-degree differential form on $\M_{0,n}$ that gives $\M_{0,n}$ the structure of a \emph{positive geometry} \cite{ArkaniHamedBaiLam2017,BrownDupont2021,ArkaniHamedHeLam2021Stringy,Lam2022}. Theorem \ref{thm:main} elicits a range of natural questions related to this subject --- for example, writing the scattering equations and Parke-Taylor form in terms of dihedral coordinates relies on a choice of a triangulation of an $n$-gon with no internal triangles; Lam asked whether such expressions exist for general triangulations.

% this paper are connected to the geometry of the real locus $\Mbar_{0,n}(\R)\subseteq\Mbar_{0,n}$, and to the theory of cluster algebras. Brown \cite{Brown2009} studied the cross-ratios in \eqref{eq:PiT} under the name \emph{dihedral coordinates}. Brown used dihedral coordinates to give affine charts on $\Mbar_{0,n}$ that are particularly well-behaved with respect to the cell decomposition of $\Mbar_{0,n}(\R)$, ultimately proving that all period integrals of $\Mbar_{0,n}$ are multiple zeta values. Dihedral coordinates have also been used to study string scattering amplitudes \cite{BrownDupont2021}, and they appear naturally as coordinates associated to type-A cluster seeds in cluster configuration spaces \cite{ArkaniHamedHeLam2021,ArkaniHamedHeLamThomas2023}. The relations between these cross-ratios, and the transition functions between collections of dihedral coordinates, are well-studied from these perspectives, and generalize well to cluster algebras associated to Dynkin diagrams of other types.

The cross-ratio degree problem appears to be ubiquitous --- special cases have been discovered and re-discovered, repeatedly and independently, by many people across a range of areas, including rigidity theory, polynomial root-finding algorithms, complex dynamics, combinatorics of matchings on bipartite graphs, birational geometry of $\Mbar_{0,n}$, and Gromov-Witten theory \cite{GalletGraseggerSchicho2020,JordanKaszanitzky2015,Reinke2022,RamadasSilversmith2020Per5,Silversmith2021,CastravetTevelev2013,Goldner2020Thesis}.

\begin{rem}
    Some of the lemmas stated in this paper apply more generally to certain natural generalizations of the cross-ratio degree problem, e.g. considering arbitrary products of forgetful maps $\M_{0,n}\to\prod_j\M_{0,S_j}$ with $\sum_j(S_j-3)=n-3$, see also \cite{BrakensiekEurLarsonLi2023}. We do not include the generalizations here as we do not know of an application.
\end{rem}

\subsection*{Idea of proof} We prove Theorem \ref{thm:main} by induction on $I(T)$, by cutting the $n$-gon into smaller polygons. Given an internal triangle $\Delta$ of $T$, we may produce three smaller polygons $T_X,T_Y,T_Z$ as in Figure \ref{fig:Decompose2}, satisfying $I(T)=I(T_X)+I(T_Y)+I(T_Z)+1$. We prove Lemma \ref{lem:Double}, which comprises most of the work of the paper --- this lemma applies more generally to cross-ratio degrees, and in this case implies $d_T=2\cdot d_{T_X}\cdot d_{T_Y}\cdot d_{T_Z}$, providing our inductive step. Our proof of Lemma \ref{lem:Double} is somewhat technical, and involves analyzing the first three iterations of a recursive algorithm for computing cross-ratio degrees first used by Goldner, using some basic lemmas on cross-ratio degrees developed by myself and others.

\begin{rem}
    A triangulation $T$ has an $[n]$-marked trivalent dual tree $\tau$, corresponding to a boundary point $P_\tau$ of $\Mbar_{0,n}$, or alternatively a top-dimensional cone $\sigma_\tau$ of the tropical moduli space $\M_{0,n}^\trop$. One can easily show that the extension of $\pi_T$ to $\Mbar_{0,n}$ is a local isomorphism at $P_\tau$, and correspondingly the tropicalized map $\pi_T^{\trop}:\M_{0,[n]}^{\trop}\to\prod_{S\in\U}\M_{0,S}^{\trop}$ maps $\sigma_\tau$ isomorphically (with determinant 1) onto a cone of the codomain. It would be interesting to understand this piece of the tropical picture better --- e.g. to classify combinatorially which \emph{other} top-dimensional cones map to $\pi_T(\sigma_\tau)$.
\end{rem}

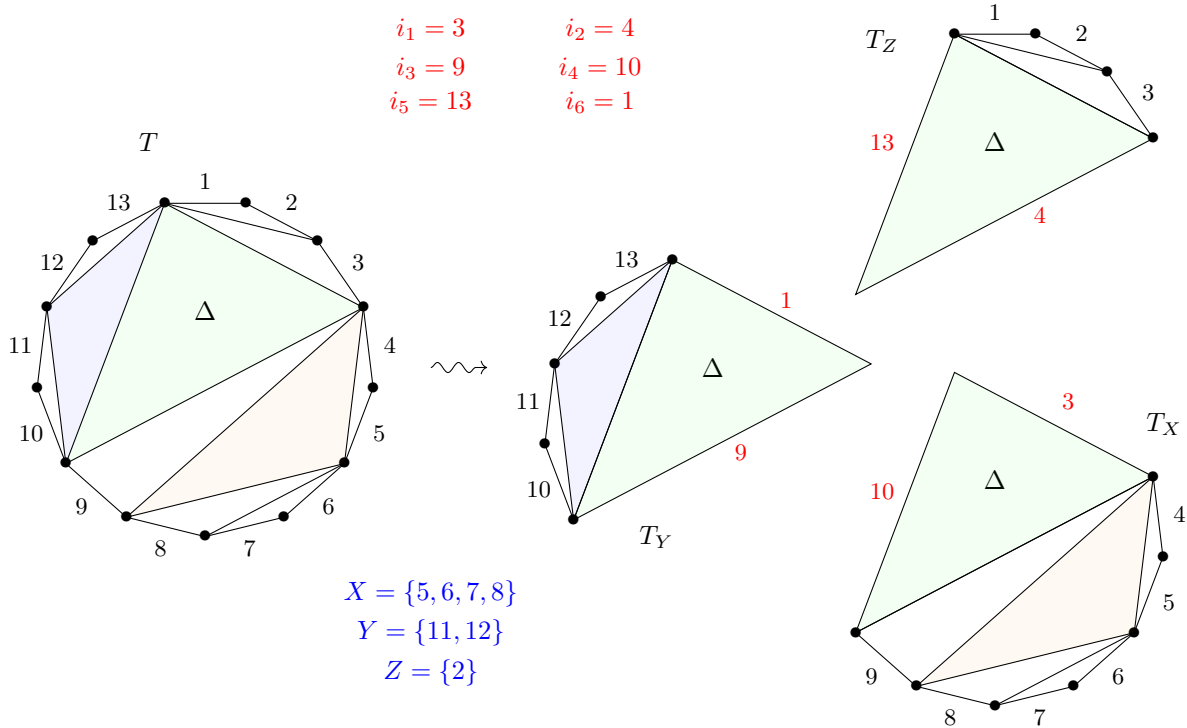
\begin{figure}
\centering
    \begin{tikzpicture}[scale=.75]
        \foreach \x in {1,...,13} {
        \draw (360*\x/13+3*360/52:3)--(360*\x/13+360/13+3*360/52:3);
        \draw (360*\x/13+3*360/52:3) node {$\bullet$};
        \draw (-360*\x/13+360*4/13+360/52:3.3) node {\small \x};
        }
        \filldraw[green,opacity=.05] (-360*0/13+15*360/52:3)--(-360*3/13+15*360/52:3)--(-360*9/13+15*360/52:3)--cycle;

        \draw (0,1) node {$\Delta$};
        \draw (9,0) node {$\Delta$};
        \draw (14,4) node {$\Delta$};
        \draw (14,-2) node {$\Delta$};
        \filldraw[blue,opacity=.05] (-360*0/13+15*360/52:3)--(-360*11/13+15*360/52:3)--(-360*9/13+15*360/52:3)--cycle;
        \filldraw[orange,opacity=.05] (-360*3/13+15*360/52:3)--(-360*5/13+15*360/52:3)--(-360*8/13+15*360/52:3)--cycle;
        \draw (-360*0/13+15*360/52:3)--(-360*3/13+15*360/52:3);
        \draw (-360*0/13+15*360/52:3)--(-360*2/13+15*360/52:3);
        \draw (-360*0/13+15*360/52:3)--(-360*11/13+15*360/52:3);
        \draw (-360*3/13+15*360/52:3)--(-360*9/13+15*360/52:3);
        \draw (-360*8/13+15*360/52:3)--(-360*3/13+15*360/52:3);
        \draw (-360*5/13+15*360/52:3)--(-360*8/13+15*360/52:3);
        \draw (-360*3/13+15*360/52:3)--(-360*5/13+15*360/52:3);
        \draw (-360*5/13+15*360/52:3)--(-360*7/13+15*360/52:3);
        \draw (-360*0/13+15*360/52:3)--(-360*9/13+15*360/52:3);
        \draw (-360*9/13+15*360/52:3)--(-360*11/13+15*360/52:3);
        \draw[->,
line join=round,
decorate, decoration={
    snake,
    segment length=8,
    amplitude=1.9,post=lineto,
    post length=2pt
}] (4,0)--(5,0);
\foreach \x in {1,...,3} {
        \draw ($(14,3)+(-360*\x/13+3*360/13+3*360/52:3)$)--($(14,3)+(-360*\x/13+4*360/13+3*360/52:3)$);
        \draw ($(14,3)+(-360*\x/13+3*360/13+3*360/52:3)$) node {$\bullet$};
        \draw ($(14,3)+(-360*\x/13+360*4/13+360/52:3.3)$) node {\small \x};
        }

        \foreach \x in {4,...,9} {
        \draw ($(14,-3)+(-360*\x/13+3*360/13+3*360/52:3)$)--($(14,-3)+(-360*\x/13+4*360/13+3*360/52:3)$);
        \draw ($(14,-3)+(-360*\x/13+3*360/13+3*360/52:3)$) node {$\bullet$};
        \draw ($(14,-3)+(-360*\x/13+360*4/13+360/52:3.3)$) node {\small \x};
        }

        \foreach \x in {10,...,13} {
        \draw ($(9,-1)+(-360*\x/13+3*360/13+3*360/52:3)$)--($(9,-1)+(-360*\x/13+4*360/13+3*360/52:3)$);
        \draw ($(9,-1)+(-360*\x/13+3*360/13+3*360/52:3)$) node {$\bullet$};
        \draw ($(9,-1)+(-360*\x/13+360*4/13+360/52:3.3)$) node {\small \x};
        }
        \draw[fill=green,fill opacity=.05] ($(14,3)+(-360*0/13+15*360/52:3)$)--($(14,3)+(-360*3/13+15*360/52:3)$)--($(14,3)+(-360*9/13+15*360/52:3)$)--cycle;

        \draw[fill=green,fill opacity=.05] ($(14,-3)+(-360*0/13+15*360/52:3)$)--($(14,-3)+(-360*3/13+15*360/52:3)$)--($(14,-3)+(-360*9/13+15*360/52:3)$)--cycle;

        \draw[fill=green,fill opacity=.05] ($(9,-1)+(-360*0/13+15*360/52:3)$)--($(9,-1)+(-360*3/13+15*360/52:3)$)--($(9,-1)+(-360*9/13+15*360/52:3)$)--cycle;
        \filldraw[blue,opacity=.05] ($(9,-1)+(-360*0/13+15*360/52:3)$)--($(9,-1)+(-360*11/13+15*360/52:3)$)--($(9,-1)+(-360*9/13+15*360/52:3)$)--cycle;
        \filldraw[orange,opacity=.05] ($(14,-3)+(-360*3/13+15*360/52:3)$)--($(14,-3)+(-360*5/13+15*360/52:3)$)--($(14,-3)+(-360*8/13+15*360/52:3)$)--cycle;
        \draw ($(14,3)+(-360*0/13+15*360/52:3)$)--($(14,3)+(-360*3/13+15*360/52:3)$);
        \draw ($(14,3)+(-360*0/13+15*360/52:3)$)--($(14,3)+(-360*2/13+15*360/52:3)$);
        \draw +($(9,-1)+(-360*0/13+15*360/52:3)$)--+($(9,-1)+(-360*11/13+15*360/52:3)$);
        \draw +($(14,-3)+(-360*3/13+15*360/52:3)$)--+($(14,-3)+(-360*9/13+15*360/52:3)$);
        \draw +($(14,-3)+(-360*8/13+15*360/52:3)$)--+($(14,-3)+(-360*3/13+15*360/52:3)$);
        \draw +($(14,-3)+(-360*5/13+15*360/52:3)$)--+($(14,-3)+(-360*8/13+15*360/52:3)$);
        \draw +($(14,-3)+(-360*3/13+15*360/52:3)$)--+($(14,-3)+(-360*5/13+15*360/52:3)$);
        \draw +($(14,-3)+(-360*5/13+15*360/52:3)$)--+($(14,-3)+(-360*7/13+15*360/52:3)$);
        \draw +($(9,-1)+(-360*0/13+15*360/52:3)$)--+($(9,-1)+(-360*9/13+15*360/52:3)$);
        \draw +($(9,-1)+(-360*9/13+15*360/52:3)$)--+($(9,-1)+(-360*11/13+15*360/52:3)$);

        \draw ($(14,3)+(3*360/13+3*360/52:3)$) node {$\bullet$};

        \draw ($(14,-3)+(-360*3/13+3*360/13+3*360/52:3)$) node {$\bullet$};

        \draw ($(9,-1)+(-360*9/13+3*360/13+3*360/52:3)$) node {$\bullet$};

        \draw (10.3,1.2) node {\small $\textcolor{red}{1}$};
        \draw (9.5,-1.5) node {\small $\textcolor{red}{9}$};

        \draw (12,4) node {\small $\textcolor{red}{13}$};
        \draw (14.8,2.7) node {\small $\textcolor{red}{4}$};

        \draw (15.3,-.6) node {\small $\textcolor{red}{3}$};
        \draw (12,-2.2) node {\small $\textcolor{red}{10}$};
        \draw[red] (4,6) node {$i_1=3$};
        \draw[red] (4,5.3) node {$i_3=9$};
        \draw[red] (4,4.7) node {$i_5=13$};
        \draw[red] (7,6) node {$i_2=4$};
        \draw[red] (7,5.3) node {$i_4=10$};
        \draw[red] (7,4.7) node {$i_6=1$};
        \draw[blue] (4,-4) node {$X=\{5,6,7,8\}$};
        \draw[blue] (4,-4.7) node {$Y=\{11,12\}$};
        \draw[blue] (4,-5.4) node {$Z=\{2\}$};
        \draw (-1,4) node {$T$};
        \draw (8,-3) node {$T_Y$};
        \draw (17,-1) node {$T_X$};
        \draw (12,5.8) node {$T_Z$};
    \end{tikzpicture}
    \caption{Decomposing the 13-gon as in the proof of Theorem \ref{thm:main}.}
    \label{fig:Decompose2}
\end{figure}

\subsection*{An open problem}
    Let $C(n)$ denote the largest cross-ratio degree on $\M_{0,n}.$ Inscribing a $\floor{n/2}$-gon inside the $n$-gon yields triangulations $T$ with $I(T)=\floor{n/2}-2$, and Theorem \ref{thm:main} then implies $C(n)\ge2^{\floor{n/2}-2}$. One may also show $C(n)\le2^{n-5}$ for $n\ge5$, see \cite[Rem. 4.3]{Silversmith2021}. What are the asymptotics of $C(n)$?
    
    The lower bound above is the best one I know of, but it is not sharp --- here is some experimental data: \begin{center}
        \begin{tabular}{c||c|c|c|c|c|c|c|c|c|c|c|c}
             $n$&3&4&5&6&7&8&9&10&11&12&13&14\\\hline
             $C(n)$&1&1&1&2&$\ge$2&$\ge$4&$\ge$6&$\ge$10&$\ge$13&$\ge$20&$\ge$28&$\ge$41
        \end{tabular}
    \end{center}
It is likely that the bounds in the table are correct for $n\le10.$ I assume it is a coincidence that the data is compatible with (a shift of) OEIS sequence A034406.

\subsection*{Acknowledgements}
This paper answers a question that Thomas Lam asked me at the Institute for Computational and Experimental Research in Mathematics (ICERM) in Providence, Rhode Island, while we were in residence for the 3-week Combinatorial Algebraic Geometry Spring 2021 Reunion Event in August 2023 (supported by NSF grant DMS-1929284). I am grateful to ICERM for their hospitality, to the organizers of the event for the opportunity to attend the wonderful program, to Thomas for bringing this class of cross-ratio degrees to my attention and explaining the context, and to Matt Larson for useful conversations. (In fact, Matt speculated that Theorem \ref{thm:main} holds.)

\section{Basic Lemmas}
We will need several standard facts about cross-ratio degrees. The first is the following vanishing condition.
\begin{lem}[\cite{Silversmith2021}, Prop. 4.1]\label{lem:Surplus}
    Let $\U\in\binom{[n]}{4}^{n-3}$. If there exists a nonempty subset $\U'\subseteq\U$ such that $\bigcup_{S\in\U'}S<\abs{\U'}+3$, then $d_{[n],\U}=0.$
\end{lem}
The following recursive algorithm for cross-ratio degrees is a straightforward application of standard facts\footnote{Explicitly, if $\rho_S:\Mbar_{0,[n]}\to\Mbar_{0,S}$ is the forgetful map that remembers the marks in $S$, then \eqref{eq:Alg} expresses the restriction of $\prod_{2\le j\le n-3}\rho_{S_j}^*[pt]$ to the sum of boundary divisors $\rho_{S_1}^*(D_{i_1i_2|i_3i_4})$, where $D_{i_1i_2|i_3i_4}$ is a boundary point in $\Mbar_{0,S_1}\cong\Mbar_{0,4}$.} about the cohomology of $\Mbar_{0,n}$. The algorithm seems to have first been written down (although somewhat implicitly) in \cite[Cor. 3.2.22]{Goldner2020Thesis}. See also \cite{GalletGraseggerSchicho2020,GriffinLevinsonRamadasSilversmith2024}. 
\begin{lem}{\cite{Goldner2020Thesis}} \label{lem:Alg} Let $\U=\{S_1,\ldots,S_{n-3}\}\in\binom{[n]}{4}^{n-3}$, and suppose $S_1=\{i_1,i_2,i_3,i_4\}.$ Then
\begin{align}\label{eq:Alg}
    d_{[n],\U}=\sum_{\substack{[n]=A_1\sqcup A_2\\i_1,i_2\in A_1\\i_3,i_4\in A_2\\\abs{S_j\cap A_1}\ne2\text{ \emph{for} }2\le j\le n-3}}d_{A_1\cup\{\star\},\U_1}\cdot d_{A_2\cup\{\dagger\},\U_2}.
\end{align} Here $\U_1$ is obtained by taking all $S\in\U$ such that $\abs{S\cap A_1}\ge 3$, and in each such $S$, replacing any element (unique if it exists) of $A_2$ with $\star.$ Similarly $\U_2$ is obtained by taking all $S\in\U$ with $\abs{S\cap A_2}\ge 3$, and replacing elements of $A_1$ with $\dagger.$ In the summand, $d_{[n],\U}$ is taken to be zero if $\abs{\U}\ne n-3$ --- in other words, we only need to consider terms where $\abs{\U_1}\in\binom{A_1\cup\{\star\}}{4}^{\abs{A_1}-2}$ and, correspondingly, $\abs{\U_2}\in\binom{A_2\cup\{\dagger\}}{4}^{\abs{A_2}-2}$.
\end{lem}
Note that $\U_1,\U_2$ are naturally identified with disjoint subsets of $\U,$ whose union is $\U\setminus\{S_1\}$.
\begin{rem}\label{rem:BuildTrees}
It is helpful to keep in mind the following pictorial interpretation of Lemma \ref{lem:Alg}. We may interpret the index set of the sum in \eqref{eq:Alg} as the set of $[n]$-marked graphs $\Gamma$ (i.e. graphs together with additional ``half-edges'' labeled by $[n]$) of the form
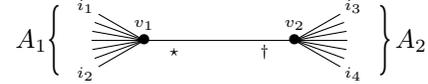
\begin{center}
    \begin{tikzpicture}
        \draw (0,0) node {$\bullet$} -- (2,0) node {$\bullet$};
        \foreach \a in {-3,...,3} { \draw (0,0)--++(180+10*\a:.7);
        \draw (2,0)--++(10*\a:.7);
        }
        \draw (-1.5,0) node {$A_1$};
        \draw (150:.9) node {\tiny $i_1$};
        \draw (210:.9) node {\tiny $i_2$};
        \draw (2,0)++(30:.9) node {\tiny $i_3$};
        \draw (2,0)++(-30:.9) node {\tiny $i_4$};
        \draw (3.6,0) node {$A_2\thickspace$};
        \draw (.4,-.15) node {\tiny $\star$};
        \draw (0,.2) node {\tiny $v_1$};
        \draw (2,.2) node {\tiny $v_2$};
        \draw (1.6,-.15) node {\tiny $\dagger$};
        \draw (-1.2,0) node {\resizebox{\width}{20pt}{$\{$}};
        \draw (3.2,0) node {\resizebox{\width}{20pt}{$\}$}};
    \end{tikzpicture}
\end{center}
where we require that for each $S_j$ with $2\le j\le n$, there is a unique vertex $v$ of $\Gamma$ such that the four paths from $v$ to the elements of $S_j$ start along four \emph{distinct} (half-)edges incident to $v$. We then say $S_j$ is supported on $v$. Note that if $S_j$ is supported on $v_1$ and $\abs{S_j\cap A_1}=3,$ then one of the four half-edges in question will be $\star$ --- this explains the ``renaming'' process in the lemma. This associates to each vertex a collection of 4-element subsets of the half-edges incident to that vertex, and the summand in \eqref{eq:Alg} corresponding to $\Gamma$ is the product of the two associated cross-ratio degrees. (We have the right \emph{number} of 4-element subsets by the assumption $\abs{\U_1}=\abs{A_1}-2.$ Note that these two cross-ratio degrees are defined on underlying sets of cardinality strictly less than $n$.

The reason for this interpretation is as follows. Suppose that for some $\Gamma,$ we apply Lemma \ref{lem:Alg} \emph{again} to expand the factor $d_{A_1\cup\{\star\},\U_1}$ in $d_{A_1\cup\{\star\},\U_1}\cdot d_{A_2\cup\{\dagger\},\U_2}$. The answer will be a sum over graphs of the form \begin{align*}
    \raisebox{-45pt}
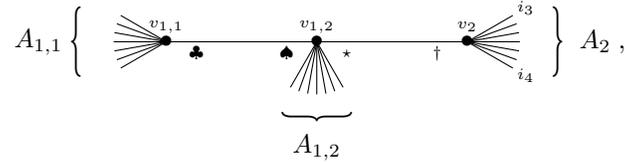
{\begin{tikzpicture}
        \draw (0,0) node {$\bullet$} -- (2,0) node {$\bullet$}-- (4,0) node {$\bullet$};
        \foreach \a in {-3,...,3} { \draw (0,0)--++(180+10*\a:.7);
        \draw (4,0)--++(10*\a:.7);
        \draw (2,0)--++(-90+10*\a:.7);
        }
        \draw (0,.2) node {\tiny $v_{1,1}$};
        \draw (2,.2) node {\tiny $v_{1,2}$};
        \draw (4,.2) node {\tiny $v_{2}$};
        \draw (-1.7,0) node {$A_{1,1}$};
        \draw (2,-1.4) node {$A_{1,2}$};
        \draw (2,-1) node {\rotatebox{270}{\resizebox{\width}{20pt}{$\}$}}};
        \draw (4,0)++(30:.9) node {\tiny $i_3$};
        \draw (4,0)++(-30:.9) node {\tiny $i_4$};
        \draw (5.8,0) node {$A_2\thickspace,$};
        \draw (0.4,-.15) node {\tiny $\clubsuit$};
        \draw (1.6,-.15) node {\tiny $\spadesuit$};
        \draw (2.4,-.15) node {\tiny $\star$};
        \draw (3.6,-.15) node {\tiny $\dagger$};
        \draw (-1.2,0) node {\resizebox{\width}{20pt}{$\{$}};
        \draw (5.2,0) node {\resizebox{\width}{20pt}{$\}$}};
    \end{tikzpicture}}
\end{align*}
where again the summand is a product of cross-ratio degrees, each consisting of subsets supported on a vertex. Iterating this process of ``splitting a vertex'' gives one strategy of computing cross-ratio degrees --- the eventual output will be a collection of trivalent trees with no nontrivial cross-ratio degrees, and the number of these trees will be $d_{[n],\U}.$
\end{rem}
\begin{rem}
    The marked trees of Remark \ref{rem:BuildTrees} can be thought of as tropical genus-zero curves --- this connection is given a full geometric explanation in \cite{GriffinLevinsonRamadasSilversmith2024}. 
\end{rem}

We will also need the following multiplicativity property.
\begin{lem}\label{lem:3Overlap}
     Let $\U\in\binom{[n]}{4}^{n-3}.$ Suppose there exists a partition $[n]=\{i_1,i_2,i_3\}\sqcup X\sqcup Y$ such that for all $S\in\U$, we have either $S\subseteq\{i_1,i_2,i_3\}\cup X$ or $S\subseteq\{i_1,i_2,i_3\}\cup Y$. Define $\U_X=\{S\in\U:S\subseteq\{i_1,i_2,i_3\}\cup X\}$, and similarly $\U_Y$. Then $$d_{[n],\U}=\begin{cases}
         d_{\{i_1,i_2,i_3\}\cup X,\U_X}\cdot d_{\{i_1,i_2,i_3\}\cup Y,\U_Y}&\abs{\U_X}=\abs{X}\\
         0&\text{otherwise}.
     \end{cases}.$$
\end{lem}
\begin{proof}
    The map $\pi_{\U}:\M_{0,[n]}\to\prod_{S\in\U}\M_{0,S}$ factors as:
    \begin{align}\label{eq:123}
        \begin{tikzcd}[ampersand replacement=\&]
            \M_{0,[n]}\arrow[r]\&\M_{0,\{i_1,i_2,i_3\}\cup X}\times\M_{0,\{i_1,i_2,i_3\}\cup Y}\arrow[r,"\pi_{\U_X}\times\pi_{\U_Y}"]\&[2em]\displaystyle\prod_{S\in\U_X}\M_{0,S}\times\prod_{S\in\U_Y}\M_{0,S}=\prod_{S\in\U}\M_{0,S}.
        \end{tikzcd}
    \end{align}
    If $\abs{\U_X}\ne\abs{X}$, then either $\pi_{\U_X}$ or $\pi_{\U_Y}$ has positive relative dimension, hence so does $\pi_\U.$ Thus $d_{[n],\U}=0.$

    Suppose $\abs{\U}=\abs{X}$. Since M\"obius transformations act simply 3-transitively on $\P^1$, we may uniquely choose coordinates on $\P^1$ so that $p_{i_1}=\infty,$ $p_{i_2}=0$, and $p_{i_3}=1$. With this choice, the distinct complex numbers $p_{i_4},\ldots,p_{i_n}\in\P^1\setminus\{\infty,0,1\}$ define global coordinates on $\M_{0,[n]}$. The leftmost map in \eqref{eq:123}, in these coordinates, is simply the inclusion of the open subset where none of the coordinates $p_i$ coincide, hence is birational. Thus the degree of $\pi_{\U}$ is equal to the degree of $\pi_{\U_X}\times\pi_{\U_Y}$, and the statement follows.
\end{proof}
\begin{cor}[Cf. {\cite[Cor. 2.19]{Brown2009}}]\label{cor:NoInternalTriangles}
    If $T$ has no internal triangles, then $d_T=1.$
\end{cor}
\begin{proof}
    Let $T$ be a triangulation of the $n$-gon with no internal triangles. We induct on $n$, with base case $n=4$. Both triangulations of a square induce the identity map $\pi_T:\Mbar_{0,[4]}\to\Mbar_{0,[4]},$ which has degree 1.

    Suppose $n\ge5.$ On average, a triangle of $T$ contains $\frac{3(n-2)-n}{n-2}>1$ diagonals. Thus there exists a triangle $\Delta$ containing at least two diagonals. Since $T$ has no internal triangles, $\Delta$ contains an edge $i_1$ of the $n$-gon. The vertex of $\Delta$ opposite $i_1$ touches two edges $i_2,i_3$, with $i_1<i_2<i_3<i_1$ in clockwise cyclic order. We construct two triangulations $T_1,T_2$ of smaller polygons as in Figure \ref{fig:Decompose1}. Specifically, let 
    \begin{align*}
        X&=\{i\in[n]:i_1<i<i_2\}&Y&=\{i\in[n]:i_3<i<i_1\}
    \end{align*}
    using clockwise cyclic order, and let $T_1$ be the triangulation of a $(\abs{X}+3)$-gon obtained by cutting $T$ along the diagonal touching edges $i_1-1,i_1,i_2,i_3$, discarding the piece not containing $\Delta,$ and renaming the cut edge by $i_2$. Similarly we have a triangulation $T_2$ of a $(\abs{Y}+3)$-gon by cutting along the diagonal touching $i_1,i_1+1,i_2,i_3$. %Let $S_1=\{i_1,i_1+1,i_2,i_3\}$ and $S_2=\{i_1-1,i_1,i_2,i_3\}$. 
    % Let \begin{align*}
    %     X&=\{i\in[n]:i_1<i<i_2\}&Y&=\{i\in[n]:i_3<i<i_1\},
    % \end{align*}
    Since diagonals cannot cross, every diagonal of $T$ either touches four sides in $\{i_1,i_2,i_3\}\cup X$ or touches four sides in $\{i_1,i_2,i_3\}\cup Y$.  There are precisely $\abs{X}$ of the first type since they triangulate a $(\abs{X}+3)$-gon. %Furthermore, the diagonals in $T$ of the first type form a triangulation $T_1$ of the $(\abs{X}+3)$-gon obtained by cutting $T$ along the edge $S_2$ (resp. $S_1$), discarding the component not containing $\Delta,$ and renaming the cut edge with $i_3$. (See Figure \ref{fig:Decompose2}; we similarly get a triangulation $T_2$ of an $(\abs{Y}+3)$-gon.) 
    Thus by Lemma \ref{lem:3Overlap}, we have %$$d_T=d_{\{i_1,i_2,i_3\}\cup X,\U_X}\cdot d_{\{i_1,i_2,i_3\}\cup Y,\U_Y},$$ where $\U_X$ and $\U_Y$ are as in the statement of the lemma. In other words, 
    $$d_T=d_{T_1}\cdot d_{T_2}.$$ % where $T_1$ (resp. $T_2$) is obtained by cutting $T$ along the edge $S_2$ (resp. $S_1$), discarding the part not containing $\Delta,$ and renaming the cut edge with $i_3$ (resp. $i_2$). (See Figure \ref{fig:Decompose1}.) 
    Note $T_1$ and $T_2$ have fewer than $n$ sides and no internal triangles. Thus by induction %since $T_1$ and $T_2$ have no internal triangles, 
    %$d_{T_1}=d_{T_2}=1$. Thus 
    $d_T=1$.
\end{proof}
\begin{figure}
\centering
    \begin{tikzpicture}[scale=.75]
        \foreach \x in {1,...,6} {
        \draw (360*\x/6+3*360/24:3)--(360*\x/6+360/6+3*360/24:3);
        \draw (360*\x/6+3*360/24:3) node {$\bullet$};
        \draw (-360*\x/6+360*4/6+360/24:3.3) node {\small \x};
        }

        \draw[fill=red,fill opacity=.05] (-360*0/6+15*360/24:3)--(-360*3/6+15*360/24:3)--(-360*2/6+15*360/24:3)--cycle;
        \draw (-360*0/6+15*360/24:3)--(-360*3/6+15*360/24:3);
        \draw (-360*0/6+15*360/24:3)--(-360*2/6+15*360/24:3);
        \draw (-360*3/6+15*360/24:3)--(-360*5/6+15*360/24:3);
        \draw (4.5,3) node {$i_1=3$};
        \draw (4.5,2.3) node {$i_2=6$};
        \draw (4.5,1.6) node {$i_3=1$};
        \draw[->,
line join=round,
decorate, decoration={
    snake,
    segment length=8,
    amplitude=1.9,post=lineto,
    post length=2pt
}] (4,0)--(5,0);

        \foreach \x in {3,...,6} {
        \draw ($(14,0)+(-360*\x/6+3*360/6+3*360/24:3)$)--($(14,0)+(-360*\x/6+4*360/6+3*360/24:3)$);
        \draw ($(14,0)+(-360*\x/6+3*360/6+3*360/24:3)$) node {$\bullet$};
        \draw ($(14,0)+(-360*\x/6+360*4/6+360/24:3.3)$) node {\small \x};
        }

        \foreach \x in {1,...,3} {
        \draw ($(9,0)+(-360*\x/6+3*360/6+3*360/24:3)$)--($(9,0)+(-360*\x/6+4*360/6+3*360/24:3)$);
        \draw ($(9,0)+(-360*\x/6+3*360/6+3*360/24:3)$) node {$\bullet$};
        \draw ($(9,0)+(-360*\x/6+360*4/6+360/24:3.3)$) node {\small \x};
        }

        \draw ($(14,0)+(-360*5/6+15*360/24:3)$)--($(14,0)+(-360*3/6+15*360/24:3)$);

        \draw ($(14,0)+(-360*2/6+15*360/24:3)$) node {$\bullet$};

        \draw ($(9,0)+(-360*0/6+3*360/6+3*360/24:3)$) node {$\bullet$};

        \draw[fill=red,fill opacity=.05] ($(14,0)+(-360*3/6+3*360/6+3*360/24:3)$)--($(14,0)+(-360*2/6+3*360/6+3*360/24:3)$)--($(14,0)+(-360*0/6+3*360/6+3*360/24:3)$)--cycle;

        \draw[fill=red,fill opacity=.05] ($(9,0)+(-360*0/6+15*360/24:3)$)--($(9,0)+(-360*2/6+15*360/24:3)$)--($(9,0)+(-360*3/6+15*360/24:3)$)--cycle;

        \draw (9.4,-0.2) node {\small $\textcolor{red}{6}$};

        \draw (12.2,.6) node {\small $\textcolor{red}{1}$};
    \end{tikzpicture}
    \caption{Decomposing along a triangle as in the proof of Corollary \ref{cor:NoInternalTriangles}.}
    \label{fig:Decompose1}
\end{figure}
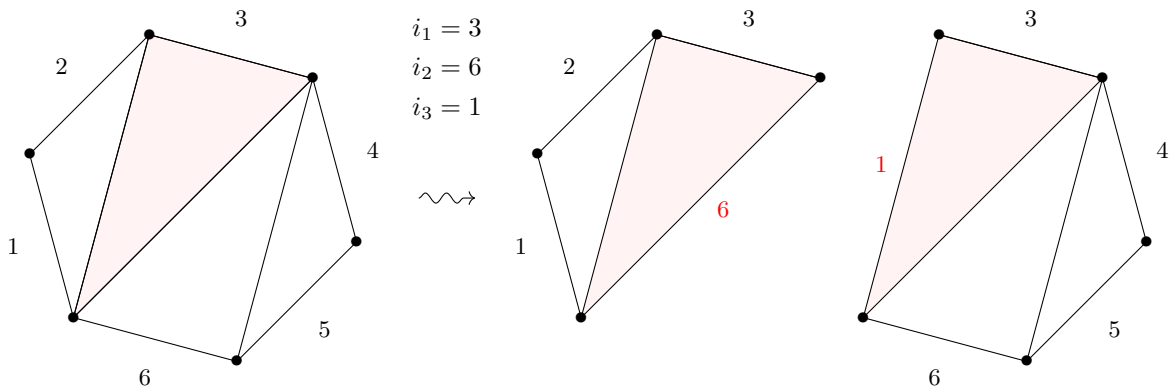
\section{Key lemma and proof of Theorem \ref{thm:main}}
Our key lemma is another (double-)multiplicativity property, which may have applications to computing other classes of cross-ratio degrees.
\begin{lem}\label{lem:Double}
    Let $\U=\{S_1,\ldots,S_{n-3}\}\in\binom{[n]}{4}^{n-3}.$ Suppose there are distinct elements $i_1,\ldots,i_6\in[n]$ with $S_1=\{i_1,i_2,i_3,i_4\},$ $S_2=\{i_3,i_4,i_5,i_6\},$ and $S_3=\{i_1,i_2,i_5,i_6\}.$ Suppose further that there exists a partition $[n]=\{i_1,\ldots,i_6\}\sqcup X\sqcup Y\sqcup Z$ such that for each $S\in\U$, we have either $S\subseteq X\cup S_1$, $S\subseteq Y\cup S_2$, or $S\subseteq Z\cup S_3$. Define $\U_X=\{S\in\U:S\subseteq X\cup S_1\},$ and similarly $\U_Y,\U_Z$.
    Then \begin{align}\label{eq:DoubleMultiplicative}
    d_{[n],\U}=2\cdot d_{X\cup S_1,\U_X}\cdot d_{Y\cup S_2,\U_Y}\cdot d_{Z\cup S_3,\U_Z}.
    \end{align}
\end{lem}
\begin{proof}
    By Lemma \ref{lem:Alg}, \begin{align}\label{eq:ApplyLemmaAlg}
    d_{[n],\U}=\sum_{\substack{[n]=A_1\sqcup A_2\\i_1,i_2\in A_1\\i_3,i_4\in A_2\\\abs{S_j\cap A_1}\ne2\text{ \emph{for} }2\le j\le n-3}}d_{A_1\cup\{\star\};\U_1}\cdot d_{A_2\cup\{\dagger\};\U_2}.
    \end{align}
    Note that the sum in \eqref{eq:ApplyLemmaAlg} \emph{excludes} partitions $[n]=A_1\sqcup A_2$ where $\{i_5,i_6\}\subseteq A_1$ (since we would then have $\abs{S_2\cap A_1}=2$) or where $\{i_5,i_6\}\subseteq A_2$ (since we would then have $\abs{S_3\cap A_2}=2$). Thus:
    \begin{align}\label{eq:BreakUpFirstTime}
    d_{[n],\U}&=\sum_{\substack{[n]=A_1\sqcup A_2\\i_1,i_2,i_5\in A_1\\i_3,i_4,i_6\in A_2\\\abs{S_j\cap A_1}\ne2\text{ \emph{for} }2\le j\le n-3}}d_{A_1\cup\{\star\};\U_1}\cdot d_{A_2\cup\{\dagger\};\U_2}+\sum_{\substack{[n]=A_1\sqcup A_2\\i_1,i_2,i_6\in A_1\\i_3,i_4,i_5\in A_2\\\abs{S_j\cap A_1}\ne2\text{ \emph{for} }2\le j\le n-3}}d_{A_1\cup\{\star\};\U_1}\cdot d_{A_2\cup\{\dagger\};\U_2}.
    \end{align}
    Using the interpretation from Remark \ref{rem:BuildTrees}, we have expressed $d_{[n],\U}$ as a sum, over graphs of the two types
    \begin{center}\raisebox{-15pt}
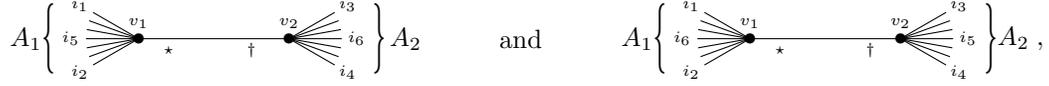
{
    \begin{tikzpicture}
        \draw (0,0) node {$\bullet$} -- (2,0) node {$\bullet$};
        \foreach \a in {-3,...,3} { \draw (0,0)--++(180+10*\a:.7);
        \draw (2,0)--++(10*\a:.7);
        }
        \draw (-1.5,0) node {$A_1$};
        \draw (150:.9) node {\tiny $i_1$};
        \draw (210:.9) node {\tiny $i_2$};
        \draw (180:.9) node {\tiny $i_5$};
        \draw (2,0)++(30:.9) node {\tiny $i_3$};
        \draw (2,0)++(-30:.9) node {\tiny $i_4$};
        \draw (2,0)++(0:.9) node {\tiny $i_6$};
        \draw (3.6,0) node {$A_2\thickspace$};
        \draw (.4,-.15) node {\tiny $\star$};
        \draw (0,.2) node {\tiny $v_1$};
        \draw (2,.2) node {\tiny $v_2$};
        \draw (1.5,-.15) node {\tiny $\dagger$};
        \draw (-1.2,0) node {\resizebox{\width}{20pt}{$\{$}};
        \draw (3.2,0) node {\resizebox{\width}{20pt}{$\}$}};
    \end{tikzpicture}}
    \quad\quad and \quad\quad\raisebox{-15pt}{
    \begin{tikzpicture}
        \draw (0,0) node {$\bullet$} -- (2,0) node {$\bullet$};
        \foreach \a in {-3,...,3} { \draw (0,0)--++(180+10*\a:.7);
        \draw (2,0)--++(10*\a:.7);
        }
        \draw (-1.5,0) node {$A_1$};
        \draw (150:.9) node {\tiny $i_1$};
        \draw (210:.9) node {\tiny $i_2$};
        \draw (180:.9) node {\tiny $i_6$};
        \draw (2,0)++(30:.9) node {\tiny $i_3$};
        \draw (2,0)++(-30:.9) node {\tiny $i_4$};
        \draw (2,0)++(0:.9) node {\tiny $i_5$};
        \draw (3.6,0) node {$A_2\thickspace,$};
        \draw (.4,-.15) node {\tiny $\star$};
        \draw (0,.2) node {\tiny $v_1$};
        \draw (2,.2) node {\tiny $v_2$};
        \draw (1.6,-.15) node {\tiny $\dagger$};
        \draw (-1.2,0) node {\resizebox{\width}{20pt}{$\{$}};
        \draw (3.2,0) node {\resizebox{\width}{20pt}{$\}$}};
    \end{tikzpicture}
    }
\end{center}
such that for each $2\le j\le n-3$, $S_j$ is supported on $v_1$ or $v_2.$ We now apply Lemma \ref{lem:Alg} to all four of the cross-ratio degrees appearing in \eqref{eq:BreakUpFirstTime}, e.g.:
    $$d_{A_1\cup\{\star\},\U_1}=\sum_{\substack{A_1\cup\{\star\}=A_{1,1}\sqcup (A_{1,2}\cup\{\star\})\\i_1,i_2\in A_{1,1}\\i_5\in A_{1,2}\\\abs{S_j\cap A_{1,1}}\ne2\text{ \emph{for} }S_j\in\U_1\text{ if }j>2}}d_{A_{1,1}\cup\{\clubsuit\};\U_{1,1}}\cdot d_{A_{1,2}\cup\{\spadesuit,\star\};\U_{1,2}}.$$ Here $\U_{1,1}$ is obtained by taking all $S\in\U_1$ with $\abs{S\cap A_{1,1}}\ge3,$ and in each such $S$, replacing any element of $A_{1,2}$ with $\clubsuit.$ Similarly $\U_{1,2}$ is obtained by taking all $S\in\U_1$ with $\abs{S\cap A_{1,2}}\ge3$, and replacing elements of $A_{1,1}$ with $\spadesuit.$ Repeating this process for all cross-ratio degrees in \eqref{eq:BreakUpFirstTime}, we see (again cf. Remark \ref{rem:BuildTrees}) that \begin{align}\label{eq:GraphContributions}
        d_{[n],\U}=\sum_\Gamma d_{A_{1,1}\cup\{\clubsuit\},\U_{1,1}}\cdot d_{A_{1,2}\cup\{\spadesuit,\star\},\U_{1,2}}\cdot d_{A_{2,1}\cup\{\dagger,\heartsuit\},\U_{2,1}}\cdot d_{A_{2,2}\cup\{\diamondsuit\},\U_{2,2}},
    \end{align} 
    where $\Gamma$ ranges over marked trees of the two types (note the positions of $i_5$ and $i_6$)
    \begin{align}\label{GraphType1}\tag{I}
    \raisebox{-45pt}
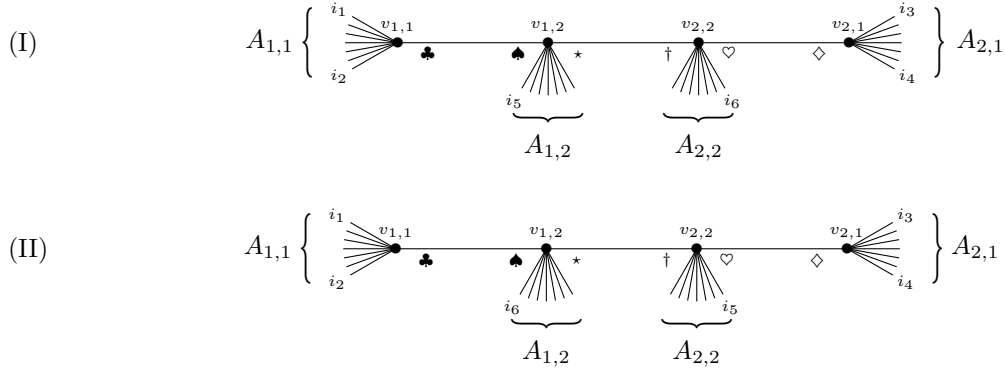
{\begin{tikzpicture}
        \draw (0,0) node {$\bullet$} -- (2,0) node {$\bullet$}-- (4,0) node {$\bullet$}-- (6,0) node {$\bullet$};
        \foreach \a in {-3,...,3} { \draw (0,0)--++(180+10*\a:.7);
        \draw (6,0)--++(10*\a:.7);
        \draw (2,0)--++(-90+10*\a:.7);
        \draw (4,0)--++(-90+10*\a:.7);
        }
        \draw (0,.2) node {\tiny $v_{1,1}$};
        \draw (2,.2) node {\tiny $v_{1,2}$};
        \draw (4,.2) node {\tiny $v_{2,2}$};
        \draw (6,.2) node {\tiny $v_{2,1}$};
        \draw (-1.7,0) node {$A_{1,1}$};
        \draw (150:.9) node {\tiny $i_1$};
        \draw (210:.9) node {\tiny $i_2$};
        \draw (2,-1.4) node {$A_{1,2}$};
        \draw (2,-1) node {\rotatebox{270}{\resizebox{\width}{20pt}{$\}$}}};
        \draw (2,0)++(-120:.9) node {\tiny $i_5$};
        \draw (4,-1) node {\rotatebox{270}{\resizebox{\width}{20pt}{$\}$}}};
        \draw (4,-1.4) node {$A_{2,2}$};
        \draw (4,0)++(-60:.9) node {\tiny $i_6$};
        \draw (6,0)++(30:.9) node {\tiny $i_3$};
        \draw (6,0)++(-30:.9) node {\tiny $i_4$};
        \draw (7.8,0) node {$A_{2,1}\thickspace$};
        \draw (0.4,-.15) node {\tiny $\clubsuit$};
        \draw (1.6,-.15) node {\tiny $\spadesuit$};
        \draw (2.4,-.15) node {\tiny $\star$};
        \draw (3.6,-.15) node {\tiny $\dagger$};
        \draw (4.4,-.15) node {\tiny $\heartsuit$};
        \draw (5.6,-.15) node {\tiny $\diamondsuit$};
        \draw (-1.2,0) node {\resizebox{\width}{20pt}{$\{$}};
        \draw (7.2,0) node {\resizebox{\width}{20pt}{$\}$}};
    \end{tikzpicture}}
\end{align}
\begin{align}\label{GraphType2}\tag{II}
    \raisebox{-45pt}{\begin{tikzpicture}
        \draw (0,0) node {$\bullet$} -- (2,0) node {$\bullet$}-- (4,0) node {$\bullet$}-- (6,0) node {$\bullet$};
        \foreach \a in {-3,...,3} { \draw (0,0)--++(180+10*\a:.7);
        \draw (6,0)--++(10*\a:.7);
        \draw (2,0)--++(-90+10*\a:.7);
        \draw (4,0)--++(-90+10*\a:.7);
        }
        \draw (0,.2) node {\tiny $v_{1,1}$};
        \draw (2,.2) node {\tiny $v_{1,2}$};
        \draw (4,.2) node {\tiny $v_{2,2}$};
        \draw (6,.2) node {\tiny $v_{2,1}$};
        \draw (-1.7,0) node {$A_{1,1}$};
        \draw (150:.9) node {\tiny $i_1$};
        \draw (210:.9) node {\tiny $i_2$};
        \draw (2,-1.4) node {$A_{1,2}$};
        \draw (2,-1) node {\rotatebox{270}{\resizebox{\width}{20pt}{$\}$}}};
        \draw (2,0)++(-120:.9) node {\tiny $i_6$};
        \draw (4,-1) node {\rotatebox{270}{\resizebox{\width}{20pt}{$\}$}}};
        \draw (4,-1.4) node {$A_{2,2}$};
        \draw (4,0)++(-60:.9) node {\tiny $i_5$};
        \draw (6,0)++(30:.9) node {\tiny $i_3$};
        \draw (6,0)++(-30:.9) node {\tiny $i_4$};
        \draw (7.8,0) node {$A_{2,1}\thickspace,$};
        \draw (0.4,-.15) node {\tiny $\clubsuit$};
        \draw (1.6,-.15) node {\tiny $\spadesuit$};
        \draw (2.4,-.15) node {\tiny $\star$};
        \draw (3.6,-.15) node {\tiny $\dagger$};
        \draw (4.4,-.15) node {\tiny $\heartsuit$};
        \draw (5.6,-.15) node {\tiny $\diamondsuit$};
        \draw (-1.2,0) node {\resizebox{\width}{20pt}{$\{$}};
        \draw (7.2,0) node {\resizebox{\width}{20pt}{$\}$}};
    \end{tikzpicture}}
\end{align}
such that for each $4\le j\le n-3,$ there is a (unique) vertex on which $S_j$ is supported. As in Lemma \ref{lem:Alg}, we have nonzero contributions only from graphs satisfying
\begin{align}\label{eq:RightCodimension}
    \abs{\U_{1,1}}&=\abs{A_{1,1}}-2&\abs{\U_{1,2}}&=\abs{A_{1,2}}-1&\abs{\U_{2,1}}&=\abs{A_{2,1}}-1&&\text{and}&\abs{\U_{2,2}}&=\abs{A_{2,2}}-2.
\end{align}
We claim types \eqref{GraphType1} and \eqref{GraphType2} each contribute $d_{X\cup S_1,\U_X}\cdot d_{Y\cup S_2,\U_Y}\cdot d_{Z\cup S_3,\U_Z}$ to $d_{[n];\U}$; this clearly implies \eqref{eq:DoubleMultiplicative}. %By symmetry we need prove this only for type \eqref{GraphType1}.

% There is a dimension condition required for a fixed graph of type \eqref{GraphType1} to give a nonzero contribution \eqref{eq:SingleGraphContribution}. For example, to have $d_{A_1\cup\{\clubsuit\};\{S_j^{\clubsuit}:S_j^\star\cap A_1\ge3\}}\ne0$ (associated to the leftmost vertex in \eqref{GraphType1}), we must have $$\abs{\{S_j^{\clubsuit}:S_j^\star\cap A_1\ge3\}}=\abs{A_1\cup\{\clubsuit\}}-3=\abs{A_1}-2.$$ Similarly, to get a nonzero contribution from a graph of type \eqref{GraphType1}, we need:
% \begin{align*}
%     \abs{\{S_j^{\spadesuit,\star}:S_j^\star\cap (A_2\cup\{\star\})\ge3\}}=\abs{A_2\cup\{\spadesuit,\star\}}-3&=\abs{A_2}-1\\
%     \abs{\{S_j^{\heartsuit,\dagger}:S_j^\star\cap (B_2\cup\{\dagger\})\ge3\}}=\abs{A_2\cup\{\heartsuit,\dagger\}}-3&=\abs{B_2}-1\\
%     \abs{\{S_j^{\diamondsuit}:S_j^\star\cap B_1\ge3\}}=\abs{A_2\cup\{\diamondsuit,\star\}}-3&=\abs{B_1}-2.
% \end{align*}

Fix a marked graph $\Gamma$ of type \eqref{GraphType1}. Recall that there is a natural injective map $\U_{1,1}\into\U$, and let $\U_{1,1;X}\subseteq\U_{1,1}$ consist of elements mapping to elements of $\U_X.$ Similarly define $U_{1,1;Y},U_{1,1;Z},U_{1,2;X},\ldots,U_{2,2;Z}.$ 
Let $A_{1,1;X}=A_{1,1}\cap X$, and similarly define $A_{1,1;Y},A_{1,1;Z},A_{1,2;X},\ldots,A_{2,2;Z}.$ 
Note the following equalities; we have similarly statements for $\U_{1,2},A_{1,2},\U_{2,2},A_{2,2},\U_{2,1},A_{2,1}.$
\begin{align}\label{eq:AUUnion}
    \U_{1,1}&=\U_{1,1;X}\sqcup\U_{1,1;Y}\sqcup\U_{1,1;Z}&
    %\U_X&=\U_{1,1;X}\sqcup\U_{1,2;X}\sqcup\U_{2,1;X}\sqcup\U_{2,2;X}\sqcup\{S_1\}\\
    A_{1,1}&=A_{1,1;X}\sqcup A_{1,1;Y}\sqcup A_{1,1;Z}\sqcup\{i_1,i_2\}.
\end{align}
In order for $\Gamma$ to contribute to $d_{[n],\U},$ the four cross-ratio degrees supported at the vertices $v_{1,1},v_{1,2},v_{2,1},v_{2,2}$ must all be nonzero. Lemma \ref{lem:Surplus} implies the following list of bounds:
\begin{enumerate}[label=(\roman*)]
    \item $\U_{1,1;X}=\emptyset$ \quad or \quad $\abs{\bigcup_{S\in \U_{1,1;X}}S}\ge\abs{\U_{1,1;X}}+3$,
    \item $\U_{1,1;Y}=\emptyset$ \quad or \quad $\abs{\bigcup_{S\in \U_{1,1;Y}}S}\ge\abs{\U_{1,1;Y}}+3$,
    \item $\U_{1,1;Z}=\emptyset$ \quad or \quad $\abs{\bigcup_{S\in \U_{1,1;Z}}S}\ge\abs{\U_{1,1;Z}}+3$,
    \item $\U_{1,2;X}\cup\U_{1,2;Y}=\emptyset$ \quad or \quad $\abs{\bigcup_{S\in \U_{1,2;X}\cup\U_{1,2;Y}}S}\ge\abs{\U_{1,2;X}\cup\U_{1,2;Y}}+3$,
    \item $\U_{1,2;Z}=\emptyset$ \quad or \quad $\abs{\bigcup_{S\in \U_{1,2;Z}}S}\ge\abs{\U_{1,2;Z}}+3$,
    \item $\U_{2,2;X}\cup\U_{2,2;Z}=\emptyset$ \quad or \quad $\abs{\bigcup_{S\in \U_{2,2;X}\cup\U_{2,2;Z}}S}\ge\abs{\U_{2,2;X}\cup\U_{2,2;Z}}+3$,
    \item $\U_{2,2;Y}=\emptyset$ \quad or \quad $\abs{\bigcup_{S\in \U_{2,2;Y}}S}\ge\abs{\U_{2,2;Y}}+3$,
    \item $\U_{2,1;X}=\emptyset$ \quad or \quad $\abs{\bigcup_{S\in \U_{2,1;X}}S}\ge\abs{\U_{2,1;X}}+3$,
    \item $\U_{2,1;Y}=\emptyset$ \quad or \quad $\abs{\bigcup_{S\in \U_{2,1;Y}}S}\ge\abs{\U_{2,1;Y}}+3$,
    \item $\U_{2,1;Z}=\emptyset$ \quad or \quad $\abs{\bigcup_{S\in \U_{2,1;Z}}S}\ge\abs{\U_{2,1;Z}}+3$.
\end{enumerate}
By definition, $\bigcup_{S\in \U_{1,1;X}}S\subseteq A_{1,1;X}\cup\{i_1,i_2,\clubsuit\}$, and $\bigcup_{S\in \U_{1,2;X}\cup\U_{1,2;Y}}S\subseteq (A_{1,2;X}\cup A_{1,2;Y})\cup\{\spadesuit,\star\}$, and so on, so the conditions above imply:
\begin{enumerate}[label=(\roman*)]\label{SurplusConditions}
    \item $\abs{A_{1,1;X}}\ge\abs{\U_{1,1;X}}$ \hspace{3in} (also holds if $\U_{1,1;X}=\emptyset$),\label{it:11X}
    \item $\abs{A_{1,1;Y}}\ge\abs{\U_{1,1;Y}}+2$ \quad or \quad $\U_{1,1;Y}=\emptyset$,\label{it:11Y}
    \item $\abs{A_{1,1;Z}}\ge\abs{\U_{1,1;Z}}$ \hspace{3in} (also holds if $\U_{1,1;X}=\emptyset$),\label{it:11Z}
    \item $\abs{A_{1,2;X}}+\abs{ A_{1,2;Y}}\ge\abs{\U_{1,2;X}}+\abs{\U_{1,2;Y}}+1$ \quad or \quad $\U_{1,2;X}\cup\U_{1,2;Y}=\emptyset$,\label{it:12XY}
    %\item $\U_{1,2;Y}=\emptyset$ \quad or \quad $\abs{A_{1,2}\cap Y}+2\ge\abs{U_{1,2;Y}}+3$,\label{it:12Y}
    \item $\abs{A_{1,2;Z}}\ge\abs{U_{1,2;Z}}$ \hspace{3in} (also holds if $\U_{1,2;Z}=\emptyset$),\label{it:12Z}
    \item $\abs{A_{2,2;X}}+\abs{A_{2,2;Z}}\ge\abs{\U_{2,2;X}}+\abs{\U_{2,2;Z}}+1$ \quad or \quad $\U_{2,2;X}\cup\U_{2,2;Z}=\emptyset$,\label{it:22XZ}
    \item $\abs{A_{2,2;Y}}\ge\abs{U_{2,2;Y}}$ \hspace{3in} (also holds if $\U_{2,2;Y}=\emptyset$),\label{it:22Y}
    \item $\abs{A_{2,1;X}}\ge\abs{U_{2,1;X}}$ \hspace{3in} (also holds if $\U_{2,1;X}=\emptyset$),\label{it:21X}
    \item $\abs{A_{2,1;Y}}\ge\abs{U_{2,1;Y}}$ \hspace{3in} (also holds if $\U_{2,1;Y}=\emptyset$),\label{it:21Y}
    \item $\abs{A_{2,1;Z}}\ge\abs{U_{2,1;Z}}+2$ \quad or \quad $\U_{2,1;Z}=\emptyset$.\label{it:21Z}
\end{enumerate}
Using \eqref{eq:RightCodimension} and \eqref{eq:AUUnion}, we have \begin{align}\label{11Comparison}
    \abs{\U_{1,1;X}}+\abs{\U_{1,1;Y}}+\abs{\U_{1,1;Z}}=\abs{\U_{1,1}}=\abs{A_{1,1}}-2=\abs{A_{1,1}\cap X}+\abs{A_{1,1}\cap Y}+\abs{A_{1,1}\cap Z},
\end{align}
and similarly
\begin{align}\label{12Comparison}
    \abs{\U_{1,2;X}}+\abs{\U_{1,2;Y}}+\abs{\U_{1,2;Z}}&=\abs{\U_{1,2}}=\abs{A_{1,2}}-1=\abs{A_{1,2}\cap X}+\abs{A_{1,2}\cap Y}+\abs{A_{1,2}\cap Z}.\\\label{22Comparison}
    \abs{\U_{2,2;X}}+\abs{\U_{2,2;Y}}+\abs{\U_{2,2;Z}}&=\abs{\U_{2,2}}=\abs{A_{2,2}}-1=\abs{A_{2,2}\cap X}+\abs{A_{2,2}\cap Y}+\abs{A_{2,2}\cap Z}.\\\label{21Comparison}
    \abs{\U_{2,1;X}}+\abs{\U_{2,1;Y}}+\abs{\U_{2,1;Z}}&=\abs{\U_{2,1}}=\abs{A_{2,1}}-2=\abs{A_{2,1}\cap X}+\abs{A_{2,1}\cap Y}+\abs{A_{2,1}\cap Z}.
\end{align}
Comparing \eqref{11Comparison} with \ref{it:11X}--\ref{it:11Z} above, we see that we must have 
\begin{align*}
    \abs{\U_{1,1;X}}&=\abs{A_{1,1;X}},&\U_{1,1;Y}&=\emptyset,&&\text{and}&\abs{\U_{1,1;Z}}&=\abs{A_{1,1;Z}}.
\end{align*} Similarly comparing \eqref{12Comparison} with \ref{it:12XY}--\ref{it:12Z}, \eqref{22Comparison} with \ref{it:22XZ}--\ref{it:22Y}, and \eqref{21Comparison} with \ref{it:21X}--\ref{it:21Z}, we find:
\begin{align*}
\U_{1,2;X}&=\emptyset,&\U_{1,2;Y}&=\emptyset,&&\text{and}&\abs{\U_{1,2;Z}}&=\abs{A_{1,2;Z}},\\
\U_{2,2;X}&=\emptyset,&\abs{\U_{2,2;Y}}&=\abs{A_{2,2;Y}},&&\text{and}&\U_{2,2;Z}&=\emptyset,\\
\abs{\U_{2,1;X}}&=\abs{A_{2,1;X}},&\abs{\U_{2,1;Y}}&=\abs{A_{2,1;Y}},&&\text{and}&\U_{2,1;Z}&=\emptyset.
\end{align*}
The contribution to \eqref{eq:GraphContributions} from $\Gamma$ is therefore $$d_{A_{1,1}\cup\{\clubsuit\},\U_{1,1;X}\cup\U_{1,1;Z}}\cdot d_{A_{1,2}\cup\{\spadesuit,\star\},\U_{1,2;Z}}\cdot d_{A_{2,2}\cup\{\dagger,\heartsuit\},\U_{2,2;Y}}\cdot d_{A_{2,1}\cup\{\diamondsuit\},\U_{2,1;X}\cup\U_{2,1;Y}}.$$ We may break this up further. Note that \eqref{eq:RightCodimension} and \eqref{eq:AUUnion} also imply $$A_{1,1;Y}=A_{1,2;X}=A_{1,2;Y}=A_{2,2;X}=A_{2,2;Z}=A_{2,1;Z}=\emptyset.$$ We thus have the decomposition $A_{1,1}\cup\{\clubsuit\}=\{i_1,i_2,\clubsuit\}\sqcup A_{1,1;X}\sqcup A_{1,1;Z}$, and any element of $\U_{1,1}=\U_{1,1;X}\sqcup\U_{1,1;Z}$ is contained by either $\{i_1,i_2,\clubsuit\}\cup A_{1,1;X}$ or $\{i_1,i_2,\clubsuit\}\cup A_{1,1;Z}$. Thus by Lemma \ref{lem:3Overlap}, we have $$d_{A_{1,1}\cup\{\clubsuit\},\U_{1,1;X}\cup\U_{1,1;Z}}=d_{A_{1,1;X}\cup\{i_1,i_2,\clubsuit\},\U_{1,1;X}}\cdot d_{A_{1,1;Z}\cup\{i_1,i_2,\clubsuit\},\U_{1,1;Z}},$$ and similarly $$d_{A_{2,1}\cup\{\diamondsuit\},\U_{2,1;X}\cup\U_{2,1;Y}}=d_{A_{2,1;X}\cup\{i_1,i_2,\diamondsuit\},\U_{2,1;X}}\cdot d_{A_{2,1;Y}\cup\{i_1,i_2,\diamondsuit\},\U_{2,1;Y}}.$$
Now, consider the contribution to \eqref{eq:GraphContributions} from \emph{all} graphs of type \eqref{GraphType1}:
\begin{align*}
    &\sum_{\Gamma\text{ type \eqref{GraphType1}}} d_{A_{1,1;X}\cup\{i_1,i_2,\clubsuit\},\U_{1,1;X}}\cdot d_{A_{1,1;Z}\cup\{i_1,i_2,\clubsuit\},\U_{1,1;Z}}\cdot d_{A_{1,2;Z}\cup\{i_5,\spadesuit,\star\},\U_{1,2;Z}}\\&\hspace{2in}\cdot d_{A_{2,2;Y}\cup\{i_6,\dagger,\heartsuit\},\U_{2,2;Y}}\cdot d_{A_{2,1;X}\cup\{i_3,i_4,\diamondsuit\},\U_{2,1;X}}\cdot d_{A_{2,1;Y}\cup\{i_3,i_4,\diamondsuit\},\U_{2,1;Y}}\\
    &\\&=\sum_{\substack{X=A_{1,1;X}\sqcup A_{2,1;X}\\Y=A_{2,2;Y}\sqcup A_{2,1;Y}\\Z=A_{1,1;Z}\sqcup A_{1,2;Z}}}d_{A_{1,1;X}\cup\{i_1,i_2,\clubsuit\},\U_{1,1;X}}\cdot d_{A_{1,1;Z}\cup\{i_1,i_2,\clubsuit\},\U_{1,1;Z}}\cdot d_{A_{1,2;Z}\cup\{i_5,\spadesuit,\star\},\U_{1,2;Z}}\\&\hspace{2in}\cdot d_{A_{2,2;Y}\cup\{i_6,\dagger,\heartsuit\},\U_{2,2;Y}}\cdot d_{A_{2,1;X}\cup\{i_3,i_4,\diamondsuit\},\U_{2,1;X}}\cdot d_{A_{2,1;Y}\cup\{i_3,i_4,\diamondsuit\},\U_{2,1;Y}}\\
    &=\left(\sum_{X=A_{1,1;X}\sqcup A_{2,1;X}}d_{A_{1,1;X}\cup\{i_1,i_2,\clubsuit\},\U_{1,1;X}}\cdot d_{A_{2,1;X}\cup\{i_3,i_4,\diamondsuit\},\U_{2,1;X}}\right)\\&\quad\quad\cdot\left(\sum_{Y=A_{2,2;Y}\sqcup A_{2,1;Y}}d_{A_{2,2;Y}\cup\{i_6,\dagger,\heartsuit\},\U_{2,2;Y}}d_{A_{2,1;Y}\cup\{i_3,i_4,\diamondsuit\},\U_{2,1;Y}}\cdot\right)\\&\quad\quad\quad\quad\cdot\left(\sum_{Z=A_{1,1;Z}\sqcup A_{1,2;Z}}d_{A_{1,1;Z}\cup\{i_1,i_2,\clubsuit\},\U_{1,1;Z}}\cdot d_{A_{1,2;Z}\cup\{i_5,\spadesuit,\star\},\U_{1,2;Z}}\right),
\end{align*}
where the sums are, as usual, over partitions $X=A_{1,1;X}\sqcup A_{2,1;X}$ such that $S\cap A_{1,1;X}\ne2$ for all $S\in\U_X$, etc. By Lemma \ref{lem:Alg}, the last expression is equal to $$d_{X\cup S_1,\U_X}\cdot d_{Y\cup S_2,\U_Y}\cdot d_{Z\cup S_3,\U_Z}$$ as desired. By symmetry, the contribution from graphs of type \eqref{GraphType2} is the same, completing the proof of the lemma.
\end{proof}

\begin{proof}[Proof of Theorem \ref{thm:main}]
    We induct on the number $I(T)$ of internal triangles in $T$, with base case $I(T)=0$ given by Corollary \ref{cor:NoInternalTriangles}.

    Fix a triangulation $T$ with $I(T)>0$. Let $\Delta$ be an internal triangle of $T$. Then $\Delta$ touches six sides of the $n$-gon, and these six sides are naturally split up into three pairs --- each pair consists of the two sides adjacent to one of the vertices of $\Delta$. Denote these pairs by $\{i_1,i_2\},$ $\{i_3,i_4\}$, and $\{i_5,i_6\}$, where $i_1<i_2<i_3<i_4<i_5<i_6<i_1$ in clockwise cyclic order. Let \begin{align*}
        X&=\{i\in[n]:i_2<i<i_3\}&Y&=\{i\in[n]:i_4<i<i_5\}&Z&=\{i\in[n]:i_6<i<i_1\},
    \end{align*}
    using the clockwise cyclic order $1\le2\le\cdots\le n\le1$. By construction, the hypotheses of Lemma \ref{lem:Double} are satisfied, and so $$d_T=2\cdot d_{X\cup S_1;\U_X}\cdot d_{Y\cup S_2;\U_Y}\cdot d_{Z\cup S_3;\U_Z}.$$ In other words, \begin{align}\label{eq:DoubleTriangulation}
        d_T=2\cdot d_{T_X}\cdot d_{T_Y}\cdot d_{T_Z},
    \end{align} where $T_X$ (resp. $T_Y,$ $T_Z$) is the triangulation of an $(\abs{X}+4)$-gon (resp. $(\abs{Y}+4)$-gon, $(\abs{Z}+4)$-gon) formed by cutting $T$ along the edges of $\Delta$ other than $S_1$ (resp. $S_2,$ $S_3$) discarding the part not containing $\Delta,$ and renaming the cut edges with $\{i_1,i_4\}$ (resp. $\{i_3,i_6\}$, $\{i_5,i_2\}$), as in Figure \ref{fig:Decompose2}. Note that $\abs{X}+4<n$, since $[n]=\{i_1,\ldots,i_6\}\sqcup X\sqcup Y\sqcup Z$, and similar $\abs{Y}+4<n$ and $\abs{Z}+4<n$. Thus we have $d_{T_X}=2^{I(T_X)}$ by the inductive hypothesis, and similarly $d_{T_Y}=2^{I(T_Y)}$ and $d_{T_Z}=2^{I(T_Z)}$. Thus \eqref{eq:DoubleTriangulation} implies $$d_T=2^{I(T_X)+I(T_Y)+I(T_Z)+1}.$$ On the other hand, every internal triangle of $T$ \textbf{except} $\Delta$ appears as an internal triangle in exactly one of $T_X,$ $T_Y,$ or $T_Z$, and all internal triangles of $T_X$, $T_Y,$ and $T_Z$ arise this way. That is, $I(T)=I(T_X)+I(T_Y)+I(T_Z)+1,$ so $d_T=2^{I(T)}.$
\end{proof}
% \begin{rem}
%     Lemmas \ref{lem:Surplus}, \ref{lem:Alg}, \ref{lem:3Overlap}, and \ref{lem:Double} all have straightforward generalizations to forgetful degrees, though we do not know of any applications as compelling as Theorem \ref{thm:main}.
% \end{rem}

\bibliography{CrossRatioEmbeddings_v2.bbl}
\bibliographystyle{amsalpha}
\end{document}